\documentclass[12pt, reqno,oneside]{amsart} 

\usepackage[utf8]{inputenc} 
\usepackage[T2A]{fontenc}
\usepackage[english]{babel}
\usepackage{amsfonts, amsmath, amscd, amssymb,amsthm}
\usepackage{xcolor}
\usepackage[all,cmtip]{xy}
\allowdisplaybreaks

\usepackage{enumitem}
\usepackage{geometry} 
\usepackage{graphicx} 
\usepackage{hyperref}
\hypersetup{pdfstartview=FitH,  linkcolor=blue,urlcolor=urlcolor, citecolor=blue, colorlinks=true}

\newtheorem{theorem}[subsection]{Theorem}
\newtheorem{lemma}[subsection]{Lemma}

\newtheorem{example}[subsection]{Example}
\newtheorem{definition}[subsection]{Definition}

\makeatletter
\@addtoreset{equation}{section}
\@addtoreset{figure}{section}
\@addtoreset{table}{section}
\makeatother

\newcommand{\str}{Z}
\newcommand\RR{\mathbb{R}}

\newcommand\NN{\mathbb{N}}

\newcommand\CC{\mathbb{C}}

\newcommand\Int[1]{\mathrm{Int}(#1)}

\newcommand{\ZZ}{\mathbb{Z}}

\newcommand{\id}{\mathrm{id}}

\newcommand\grp[2]{\varPi_1(#1,#2)}
\newcommand\grpXXp{\grp{\oX}{\qX}}
\newcommand\grpYYp{\grp{\oY}{\qY}}
\newcommand\grpXlXpl{\grp{\oX_{\lambda}}{\qX_{\lambda}}}

\newcommand\CIX{C\bigl( (I,\partial I), (\oX, \qX)\bigr)}
\newcommand\ecl[1]{[#1]}

\newcommand\eprj[1]{r_{#1}}
\newcommand\eprjX{\eprj{\oX}}
\newcommand\eprjY{\eprj{\oY}}

\newcommand\incl[1]{j_{#1}}
\newcommand\inclGrp[1]{\Pi_1(\incl{#1})}

\newcommand\UnionLambda{\mathop{\sqcup}\limits_{\lambda\in\Lambda}}
\newcommand\UnionJK{\mathop{\sqcup}\limits_{(j,k)\in\Lambda^2}}
\newcommand\UnionA{\mathop{\sqcup}\limits_{\alpha\in A}}

\newcommand\cutset{\str'}
\newcommand\cutGset{G'}

\newcommand\eps{\varepsilon}

\newcommand\catObj[1]{\mathrm{Ob}(#1)}
\newcommand\catHom[3]{\mathrm{Hom}_{#1}(#2,#3)}

\newcommand\catC{\mathfrak{C}}
\newcommand\catCObj{\catObj{\catC}}
\newcommand\catCHom[2]{\catHom{\catC}{#1}{#2}}

\newcommand\catD{\mathfrak{D}}
\newcommand\catDObj{\catObj{\catD}}
\newcommand\catDHom[2]{\catHom{\catD}{#1}{#2}}

\newcommand\oX{X}
\newcommand\oY{Y}
\newcommand\oZ{Z}
\newcommand\oW{W}

\newcommand\qX{\oX'}
\newcommand\qY{\oY'}

\newcommand\Fol{\mathcal{F}}
\newcommand\leaf{\gamma}
\newcommand\sat[1]{\mathrm{Sat}(#1)}
\newcommand\nb[1]{N_{#1}}
\newcommand\seam[1]{\omega_{#1}}
\newcommand\nbs[1]{\nb{\seam{#1}}}
\newcommand\nbStr[1]{\nb{S_{#1}}}

\newcommand\lb[1]{L_{#1}}
\newcommand\lbs[1]{\lb{#1}}
\newcommand\lbStr[1]{\lb{s_{#1}}}

\newcommand\Ucov[1]{U_{#1}}
\newcommand\zUcov[1]{U'_{#1}}

\newcommand\Vcov[1]{V_{#1}}
\newcommand\zVcov[1]{V'_{#1}}

\newcommand\grpUcov[1]{\grp{\Ucov{#1}}{\zUcov{#1}}}
\newcommand\grpVcov[1]{\grp{\Vcov{#1}}{\zVcov{#1}}}
\newcommand\coverU{\mathcal{U}}
\newcommand\coverV{\mathcal{V}}

\newcommand\indxx[1]{#1_{2}}
\newcommand\indx[1]{#1_{1}}

\makeatletter
\newcommand\testshape{family=\f@family; series=\f@series; shape=\f@shape.}
\def\myemphInternal#1{\if n\f@shape%
  \begingroup\itshape #1\endgroup\/%
  \else\begingroup\sffamily #1\endgroup%
  \fi}
\def\myemph{\futurelet\testchar\MaybeOptArgmyemph}
\def\MaybeOptArgmyemph{\ifx[\testchar \let\next\OptArgmyemph
  \else \let\next\NoOptArgmyemph \fi \next}
\def\OptArgmyemph[#1]#2{\index{#1}\myemphInternal{#2}}
\def\NoOptArgmyemph#1{\myemphInternal{#1}}
\makeatother

\title{Fundamental groupoids and homotopy types of non-compact surfaces}
\author{Sergiy Maksymenko}
\address{Institute of Mathematics of NAS of Ukraine}
\email{maks@imath.kiev.ua}
\author{Oleksii Nikitchenko}
\address{National Technical University of Ukraine
``Igor Sikorsky Kyiv Polytechnic Institute''}
\email{legionsun@gmail.com}

\begin{document}
\maketitle

\begin{abstract}
The paper contains an application of van Kampen theorem for groupoids to computation of homotopy types of certain class of non-compact foliated surfaces obtained by at most countably many strips $\mathbb{R}\times(0,1)$ with boundary intervals in $\mathbb{R}\times\{\pm1\}$ along some of those intervals.
\end{abstract}

\section{Introduction}
The present paper is devoted to applications of van Kampen theorem for groupoids to computation of homotopy types of a certain class of non-compact foliated surfaces called \emph{striped surfaces}.

It was mentioned by M.~Morse that a smooth function $f$ with non-degenerate critical points (a Morse function) on a compact manifold $\str$ ``contains'' a lot of homological information about the manifold itself (the famous Morse inequalities).
In particular, if $\dim M = 2$, so $\str$ is a ``surface'', one can even determine the topological type of $\str$ by the numbers of critical points of distinct indices via any Morse function $f:M \to \RR$.
Motivated by study functions of complex variable, in particular, harmonic functions being real and imaginary parts of holomorphic functions, Morse extended his observations in the book~\cite{Morse:TopMeth:1947} to pseudoharmonic functions $f$ defined on compact domains $\str$ in the complex plane $\CC$.

By definition, a pseudoharmonic function $f:\CC \supset \str\to\RR$ is locally a composition $g\circ h$, where $g$ is a harmonic function, and $h$ is a homeomorphism of $\CC$.
Such a function is continuous, all its critical (in a proper sense) points belonging to the interior of $\str$ are isolated and are not local extremes (due to maximum principle for holomorphic functions).
Moreover, it is assumed that the restriction of $f$ to $\partial M$ has only finitely many local minimums and maximums.
Absence of local extremes in the interior of $\str$ implies (by Jordan curve theorem) that the foliation of $\str$ into connected components of level sets of $f$ has no closed curves.

W.~Kaplan~\cite{Kaplan:DJM:1940, Kaplan:DJM:1941} characterized such foliation for pseudoharmonic function on $\RR^2$.
Namely, he shown that every foliation $\Fol$ on the plane $\RR^2$ has the following properties:
\begin{enumerate}[label={\rm(\arabic*)}]
\item\label{enum:Kaplan:1}
every leaf $\omega$ of $\Fol$ is an image of a proper embedding $\omega:\RR \to \RR^2$, so $\lim\limits_{t\to\pm\infty}\omega(t) = \infty$;

\item\label{enum:Kaplan:2}
there exists at most countably many leaves $\{\omega_i\}_{i\in\Lambda}$ of $\Fol$, such that $\RR^2\setminus\mathop{\cup}\limits_{i\in\Lambda} \omega_i$ is a disjoint union of ``open strips''  $\RR\times(0,1)$ foliated into lines $\RR\times t$, $t\in(0,1)$;

\item\label{enum:Kaplan:3}
there exists a pseudoharmonic function $f:\RR^2\to\RR$ without critical points such that $\Fol$ is a partition of $\RR^2$ into connected components of level-sets of $f$.
\end{enumerate}

Property~\ref{enum:Kaplan:3} is in a spirit of de Rham theory: if $\Fol$ is smooth, then it is defined by some closed differential $1$-form, and since $\RR^2$ is contractible, $\omega = df$ for some function $f$ satisfying~\ref{enum:Kaplan:3}.

It also gives a certain connection between foliations on surfaces and pseudoharmonic functions on the plane.
Let $\Fol'$ be a foliation on a connected surface $\str'$ without boundary distinct from $2$-sphere and projective plane, and $p:\str\to\str'$ be the universal covering map.
Then we get a well-defined foliation $\Fol$ on $\str$ whose leaves are connected components of the inverses under $p$ of leaves of $\Fol'$, and the group $\pi_1\str'$ of covering transformations interchanges the leaves of $\Fol$. 
By~\cite[Corollary~1.8]{Epstein:AM:1966} $\str$ is homeomorphic to $\RR^2$, whence $\Fol$ satisfies~\ref{enum:Kaplan:1}-\ref{enum:Kaplan:3}, while $\Fol'$ may loose all of those properties.
One may say that $\Fol'$ is obtained from a foliation on $\RR^2$ into connected components of level sets of some pseudoharminic function by some free and properly discontinuous action of $\pi_1\str$.
In general, such a function is not invariant with respect to the action of $\pi_1\str$.

Notice also that property~\ref{enum:Kaplan:2} proposes a certain classification of foliations on $\RR^2$ by studying the way in which such a foliation is glued from open strips.
This ``combinatorics of gluing'' would give certain topological invariants of (pseudo)harmonic functions and potentially information about symmetries of differential equations whose coefficients are harmonic functions.
Moreover, one could try to extend Kaplan's technique to arbitrary pseudoharmonic functions $f:\RR^2\to\RR$ just by removing the set of singular points $\Sigma_f$ of $f$ and consider the remained foliation on $\RR^2\setminus\Sigma_f$.
Such an approach was used in the papers by W.~Boothby~\cite{Boothby:AJM_1:1951, Boothby:AJM_2:1951}, M.~Morse and J.~Jenkins~\cite{JenkinsMorse:AJM:1952}, M.~Morse~\cite{Morse:FM:1952}.

Kaplan tried to minimize the total number of leaves, and for that reason his construction was not ``canonical''.
Namely, the choice of leaves $\{\omega_i\}_{i\in\Lambda}$ (and thus a ``cutting of the foliation $\Fol$'' into open strips) was not unique.
Moreover, if $S$ is a connected component of $\RR^2\setminus\mathop{\cup}\limits_{i\in\Lambda} \omega_i$, then the homeomorphism $S \to \RR\times(0,1)$ does not necessarily extend to an embedding $\overline{S} \to \RR\times[0,1]$.

In a series of papers by S.~Maksymenko, Ye.~Polulyakh~\cite{MaksymenkoPolulyakh:PIGC:2015, MaksymenkoPolulyakh:PIGC:2016, MaksymenkoPolulyakh:MFAT:2016, MaksymenkoPolulyakh:PIGC:2017} and Yu.~Soroka \cite{Soroka:MFAT:2016, Soroka:UMJ:2017} it was studied a class of foliations on non-compact surfaces $\str$ (called striped) glued from open strips in a certain ``canonical'' way.
In a joint paper~\cite{MaksymenkoPolulyakhSoroka:PIGC:2017} of the above three authors it was also described an analogue of mapping class group for foliated homeomorphisms (sending leaves into leaves) of such foliations $\Fol$ and proved that it is isomorphic with an automorphism group of a certain graph (one-dimensional CW-complex) $G$ encoding an information about gluing a surface from strips.
This graph is in a certain sense \emph{dual} to the space of leaves of $\Fol$ .

The aim of the present paper is to prove that the connection between a striped surface $\str$ and its graph $G$ is more deep: namely they are homotopy equivalent, see Theorems~\ref{th:phi_Pi1_iso} and~\ref{th:phi_Pi1_iso_groupoids}.
One of the difficulties of proving such a result is that there is no canonical map $\str \to G$.
We construct a continuous injection $\varphi: G \to \str$, and then prove (using van Kampen theorem for groupoids established by R.~Brown and A.~Salleh~\cite{BrownSalleh:AM:1984}) that $\varphi$ induces an isomorphism of fundamental groupoids of $G$ and $\str$.
This will imply that $\varphi$ is a homotopy equivalence, since $G$ and $\str$ are aspherical.
For instance the ranks of their homology groups $H_1(\str,\ZZ)$ and $H_1(G,\ZZ)$ are the same.

In fact, the result is rather simple when $\str$ is glued of finitely many strips: in this case the above map $\varphi: G \to \str$ is an embedding and its image $\varphi(G)$ is a strong deformation retract of $\str$.
On the other hand, if the number of strips is infinite, $G$ can be not a locally finite CW-complex, having thus no countable local bases at some vertices.
Therefore there will be no embeddings of $G$ into a manifold $\str$, thus the image $\varphi(G)$ will not be a strong deformation retract of $\str$.
Nevertheless, van Kampen theorem allows to accomplish the result.

Actually, the obtained result has no deal with a foliation itself but only with a way in which a surface is glued from strips.
Nevertheless, suppose we are given a foliation $\Fol$ on a non-compact surface $\str$ whose leaves are non-compact closed subsets of $\str$.
Now, if $\Fol$ has ``not so much singular leaves'', see Theorem~\ref{th:char_strip_surf} below and Figure~\ref{fig:xy}, then it is a striped surface.
Therefore we get a partition into strips and our result shows that the foliation ``contains'' an information about the homotopy type of the underlying surface.
Thus our statement could be viewed in the frame of Morse theory in which the gradient lines connecting critical points of a Morse function $f:\str\to\RR$ (or equivalently decomposition of $\str$ into handles in the sense of S.~Smale corresponding to those critical points) determine a CW-partition of $\str$ and this relates $f$ with homological and even topological structure of $\str$.

The exposition of the paper is intended to be elementary in order to make it accessible to a large audience of readers, and thus to propagate and popularize usage of  homotopy methods (like van Kampen theorem for groupoids) to more applied problems.

\section{Striped surface and its graph}

\begin{definition}[\cite{MaksymenkoPolulyakh:PIGC:2015}]
A subset $S \subset \RR \times [-1,1]$ will be called a \myemph{model strip} if:
\begin{enumerate}[label={\rm\arabic*)}]
\item\label{enum:strip_surf:1} $\RR \times (-1,1) \subset S$;
\item\label{enum:strip_surf:2} the intersection $S\cap(\RR \times \{\pm1\})$ is a union of at most countably many mutually disjoint open intervals.
\end{enumerate}
\end{definition}
For instance, $\RR \times (-1,1)$, $\RR \times [-1,1]$,  $\bigl( \RR \times (-1,1) \bigr) \cup \bigl( (-2,3)\times \{1\} \bigr)$ are model strips.
Of course one can replace $[-1,1]$ with any other closed segment $[a,b]\subset\RR$.

Notice that condition~\ref{enum:strip_surf:2} is equivalent to the assumption that $S$ is open as a subset of $\RR \times [-1,1]$.
Define the following subsets of $S$:
\begin{align*}
\ \partial_{-}S &:= S\cap(\RR \times \{-1\}), &
\ \partial_{+}S &:= S\cap(\RR \times  \{1\}),  &
\ \partial S    &:= \partial_{-}S\cup \partial_{+}S.
\end{align*}

\begin{definition}[\cite{MaksymenkoPolulyakh:PIGC:2015}]
Let $\str$ be a two-dimensional manifold.
A \myemph{striped atlas} on $\str$ is a map $q: \str_0 \to \str$ having the following properties:
\begin{enumerate}[label={\rm(\arabic*)}, leftmargin=*, itemsep=1ex]
\item $\str_0 = \UnionA S_{\alpha}$ is at most countable disjoint union of model strips;
\item $q$ is a quotient map, so a subset $U\subset\str$ is open iff $q^{-1}(U)$ is open in $\str_0$;
\item there are two disjoint families $X = \{X_\beta\}_{\beta \in B}$, $Y = \{Y_\beta\}_{\beta \in B} \subset \bigcup\limits_{\alpha \in A}\partial S_{\alpha}$ of boundary intervals such that:
\begin{enumerate}[label={\rm\alph*)}, itemsep=1ex, topsep=1ex]
\item $q$ is an injective on $\str_0 \setminus (X\cup Y)$;
\item $q(X_\beta)=q(Y_\beta)$ for each $\beta \in B$ and the restrictions $q|_{X_\beta}: X_\beta \to q(X_\beta)$ and $q|_{Y_\beta}: Y_\beta \to q(Y_\beta)$ are embeddings;
\item $q(X_\beta)\cap q(X_{\beta^{\prime}})= \varnothing$ for each $\beta \neq \beta^\prime \in B$;
\item $q(X_\beta)\cap q(\str_0 \setminus (X\cup Y))=\varnothing$ for each $\beta \in B$.
\end{enumerate}
\end{enumerate}
\end{definition}

Figures~\ref{fig:striped_atlas_2}, \ref{fig:xy}, \ref{fig:three_strips} contain examples of striped atlases.

The pair $(\str, q)$ will also be called a \myemph{striped structure} on $\str$.
A \myemph{striped surface} is a surface admitting a striped atlas.
When talking about a striped surface $\str$ we will also assume that some striped atlas $(\str, q)$ is fixed.
Notice that a striped surface is a non-compact two-dimensional manifold which can be non-orientable and disconnected.
Moreover, each connected component of its boundary is an open interval.

\subsection{Seams}
Let $\beta \in B$.
Then we have a homeomorphism $\gamma_\beta: Y_\beta \to X_\beta$ given by $\gamma_\beta = (q|_{X_\beta})^{-1}\circ q|_{Y_\beta}$.
Therefore, a striped surface $\str$ can also be regarded as a quotient space obtained by gluing some pairs of boundary intervals of model strips via homeomorphisms $\{\gamma_\beta\}_{\beta \in B}$.
The image
\[
    \seam{\beta} := q(X_{\beta})=q(Y_{\beta})
\]
will be called a \myemph{seam} of $\str$ (as well as of $q$).
In Figure~\ref{fig:striped_atlas_2} seams are colored in red color.

Since $q^{-1}(\seam{\beta}) = X_{\beta} \cup Y_{\beta}$ is closed in $\str_0$, and $q$ is a quotient map, it follows that each seam is a closed subset of $\str$.

\begin{figure}[htbp!]
\includegraphics[width=0.9\textwidth]{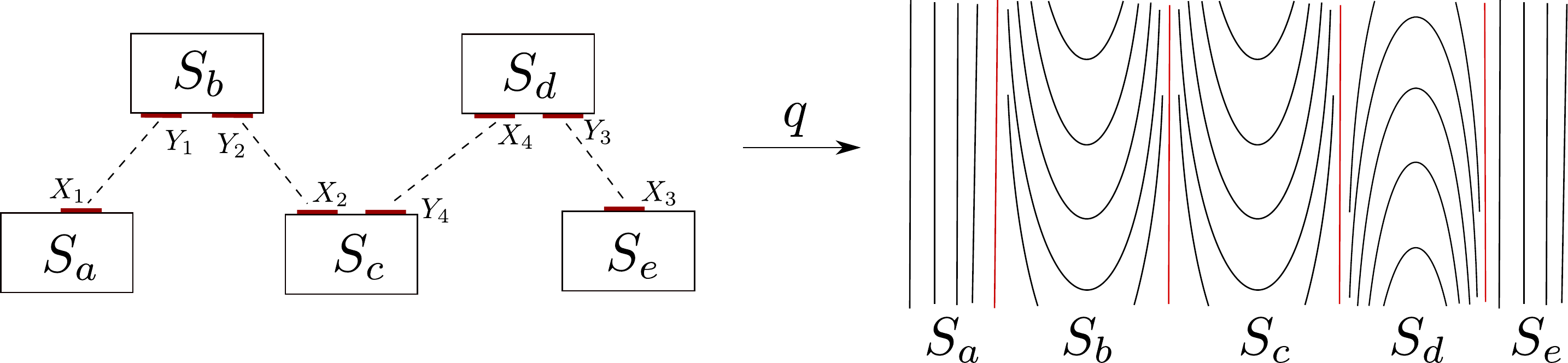} \\ (a)
\\[4mm]
\includegraphics[width=0.9\textwidth]{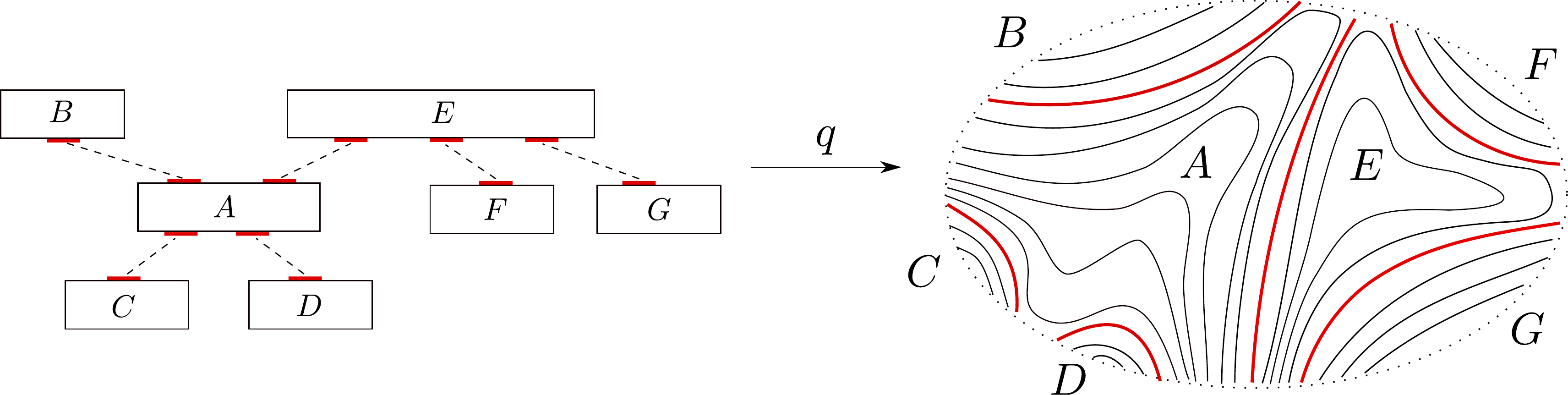} \\  (b)
\caption{Striped atlases}\label{fig:striped_atlas_2}
\end{figure}

\subsection{Foliated characterization of striped surfaces}
Though the notion of a striped surface looks rather restrictive, it nevertheless covers a large class of surfaces equipped with foliations which agree in a certain sense with a decomposition into strips.
We recall here the principal statement from~\cite{MaksymenkoPolulyakh:PIGC:2017}.
It will not be used for the proof of our main result, however we present it to describe the general picture.

By a foliated surface $(\str,\Fol)$ we will mean a two-dimensional manifold $\str$ equipped with a foliation $\Fol$.
A \myemph{saturation} $\sat{U}$ of a subset $U\subset \str$ is the union of all leaves of $\Fol$ intersecting $U$, that is
\[
\sat{U} := \mathop{\cup}_{\gamma\cap U\not=\varnothing} \gamma
\]
A subset $U \subset \str$ is \myemph{saturated} whenever $U = \sat{U}$ i.e. $U$ is a union of leaves of $\Fol$.

Given an open subset $U\subset Z$, denote by $\Fol|_{U}$ the foliation on $U$ consisting of connected components of the intersections of leaves of $\Fol$ with $U$.
We will call $\Fol|_{U}$ the \myemph{restriction of $\Fol$} to $U$.

A homeomorphism $h:\str \to \str'$ between foliated surfaces $(\str,\Fol)$ and $(\str\,\Fol')$ is \myemph{foliated} whenever for each leaf $\leaf$ of $\Fol$ its image $h(\leaf)$ is a leaf of $\Fol'$.

Notice that each model strip $S$ admits a natural foliation into boundary intervals and lines $\RR\times\{t\}$, $t\in(-1,1)$.
We will call this foliation \myemph{canonical}.
Now if $(\str,q)$ is a striped surface, then canonical foliations on the corresponding model strips induce a foliation on all of $\str$ which we will also call the \myemph{canonical foliation (of the striped atlas $q$)}.

Let $(\str,\Fol)$ be a foliated surface.
Say that a leaf $\leaf \subset \Int{\str}$, resp.\ $\leaf \subset \partial\str$, of $\Fol$ is \myemph{regular} if there exists an open saturated neighborhood $U$ of $\gamma$ such that the pair $(\overline{U}, U)$ is foliated homeomorphic to the model strip $\bigl( \RR\times[-1,1], \RR\times(-1,1) \bigr)$,  resp.~$\bigl( \RR\times[0,1], \RR\times[0,1) \bigr)$, via a homeomorphism sending $\leaf$ to $\RR\times 0$.

A leaf which is not regular will be called \myemph{singular}.
For example, in the above figures, the seams (red leaves) are precisely singular leaves.

For the case of striped surfaces and its canonical foliation, it is easy to see that a leaf $\leaf$ is singular if and only if it satisfies one of the following conditions:
\begin{enumerate}[label={\rm\roman*)},leftmargin=*]
\item\label{enum:sing:1}
$\leaf = q(\delta)$, for some boundary interval $\delta = (a,b)\times \eps 1 \subset \partial_{\eps}S_{\alpha}$, where $\eps\in\{\pm\}$ and $\alpha \in A$;

\item\label{enum:sing:2}
$\partial_{\eps}S_{\alpha}$ contains some other boundary intervals distinct from $\delta$.
\end{enumerate}

\begin{theorem}[{\cite[Theorem~4.4]{MaksymenkoPolulyakh:PIGC:2017}}]\label{th:char_strip_surf}
Let $(\str,\Fol)$ be a foliated surface such that every leaf $\leaf$ of $\Fol$ is a non-compact closed subset of $\str$.
Then the following conditions are equivalent:
\begin{enumerate}[label={\rm\arabic*)}, leftmargin=*]
\item
$\str$ admits a striped atlas $q$ such that $\Fol$ coincides with the canonical foliation of $q$;

\item
the collection of all singular leaves of $\Fol$ is locally finite.
\end{enumerate}
\end{theorem}

In fact, not every foliation on the plane admits a striped atlas, see~\cite[Examples~7.6, 7.7]{MaksymenkoPolulyakhSoroka:PIGC:2017}.
One of the reasons is that seams can converge to other seams.
We will briefly recall \cite[Example~7.6]{MaksymenkoPolulyakhSoroka:PIGC:2017}.
Let $\Fol$ be the foliation on $\RR^2$ into parallel lines $\RR\times t$, $t\in\RR$.
Let also $\{a_n\}_{n\in\NN} \RR$ be a strictly decreasing sequence of reals converging to some $a\in \RR$, and $K = \{a\} \cup \{a_n\}_{n\in\NN}$.
Consider the open subset $\str = \RR^2 \setminus (0\times K)$ and let $\Fol|_{\str}$ be the restriction of $\Fol$ to $\str$.
Notice that every point $b \in K$ splits $\RR\times b$ into two arcs $\leaf^{-}_b = (-\infty, 0) \times b$ and $\leaf^{+}_b = (0, +\infty) \times b$ being the leaves of $\Fol|_{\str}$.
By property~\ref{enum:sing:2} the latter leaves are singular for $\Fol|_{\str}$ and they converse to the leaves $\leaf^{-}_a$ and $\leaf^{+}_a$.
Hence the family of all singular leaves of $\Fol|_{\str}$ is not locally finite, and by Theorem~\ref{th:char_strip_surf} $(\str,\Fol|_{\str})$ does not admit a striped atlas.

\subsection{Graph of a striped surface}
The above figures propose to consider a graph which encodes the gluing of strips.
Such a graph was introduced in~\cite[\S6]{MaksymenkoPolulyakhSoroka:PIGC:2017} and also takes to account boundary components of striped surface which are not seams.
We will consider here a more simplified version of it.

Let $q: \UnionA  S_{\alpha} \to\str$ be a striped atlas on a surface $\str$.
Then one can associate to $q$ a one-dimensional CW-complex (``topological graph'') $G$ whose vertices are strips of the atlas and the edges correspond to gluing strips along boundary intervals.
More precisely,
\begin{enumerate}[label={\rm\arabic*)}, start=0]
\item
$0$-skeleton of $G$ is $G^{(0)} = A$;

\item
Let $\beta\in B$, so we have a pair of boundary intervals $\{X_\beta, Y_\beta \}$ such that $X_\beta \subset \partial S_{\alpha}$ and $Y_\beta \subset \partial S_{\alpha'}$ for some \myemph{vertices} $\alpha,\alpha'\in A$.
Let also $I_{\beta}=[-1,1]$.
Glue $I_{\beta}$ to $A$ via the following map:
\[
    \chi_{\beta}: \partial I_{\beta}= \{\pm1\} \to A,
    \qquad
    \chi_{\beta}(-1)=\alpha,
    \qquad
    \chi_{\beta}(1)=\alpha'.
\]
\end{enumerate}
Then the resulting CW-complex:
\[
G = \Bigl(\, \mathop{\sqcup}\limits_{\beta\in B} I_{\beta} \Bigr) \bigcup_{\chi_{\beta}, \ \beta\in B}  A.
\]
will be called the \myemph{graph of the striped atlas $q$}.
We will also denote by the same letter $\chi_{\beta}: I_{\beta} \to G$ the characteristic map of the $1$-cell $I_{\beta}$.
Thus
\begin{itemize}[leftmargin=5ex]
\item
$\chi_{\beta}|_{(-1,1)}: (-1,1) \to G$ is an embedding and we will denote by $e_{\beta} = \chi_{\beta}\bigl( (-1,1) \bigr)$ its image being an open $1$-cells;
\item
$\chi_{\beta}|_{\partial I_{\beta}}: \partial I_{\beta} \to A$ is the corresponding gluing map.
\end{itemize}

\subsection{Canonical injection $\varphi:G \to \str$}\label{sect:inject_G_Z}

We will construct here a continuous injective map $\varphi:G \to \str$.

For each $\alpha \in A$ let $s_{\alpha} := (0, 0) \in \RR\times(-1,1) \subset S_{\alpha}$ be the origin.
Define the map $\varphi:G^{(0)} \to \str$ by $h(\alpha) = s_{\alpha}$.

Furthermore, for each $\beta\in B$ fix a point $z_{\beta} \in \seam{\beta} = q(X_{\beta}) = q(Y_{\beta})$, and let
\begin{align}\label{equ:xbeta_ybeta}
    (x_{\beta}, \eps) &= q^{-1}(z_{\beta}) \cap X_{\beta},  &
    (y_{\beta}, \eps') &= q^{-1}(z_{\beta}) \cap Y_{\beta},
\end{align}
where $\eps,\eps'\in\{\pm1\}$.
Assume that $X_{\beta} \subset S_{\alpha}$ and $Y_{\beta} \subset S_{\alpha'}$ for some $\alpha,\alpha'\in A$.

Now define the following path $\varphi_{\beta}:I_{\beta} \equiv [-1,1] \to \str$ by the following formula. see Figure~\ref{fig:path_h}:
\[
\varphi_{\beta}(t) =
\begin{cases}
    q\bigl( 2(1+t) x_{\beta}, (1+t) \eps \bigr),  &  t \in [-1,-\tfrac{1}{2}], \\
    q\bigl( x_{\beta},        (1+t) \eps \bigr),  &  t \in [-\tfrac{1}{2}, 0], \\
    q\bigl( y_{\beta},        (1-t) \eps' \bigr), &  t \in [0, \tfrac{1}{2}], \\
    q\bigl( 2(1-t)y_{\beta},  (1-t) \eps' \bigr), &  t \in [\tfrac{1}{2}, 1].
\end{cases}
\]
It easily follows that
\begin{align*}
\varphi_{\beta}(-1) &=  h(\alpha) = s_{\alpha}, &
\varphi_{\beta}(0)  &= z_{\beta}, &
\varphi_{\beta}(1)  &=  h(\alpha') = s_{\alpha'}.
\end{align*}

Moreover, $\varphi_{\beta}$ is a simple path if $\alpha\not=\alpha'$, and a simple closed path if $\alpha=\alpha'$.

Since the paths $\varphi_{\beta}$ agree with $\varphi$ at end-points, they induce the required map $\varphi:G \to \str$.

\begin{figure}[htbp!]
\includegraphics[width=0.8\textwidth]{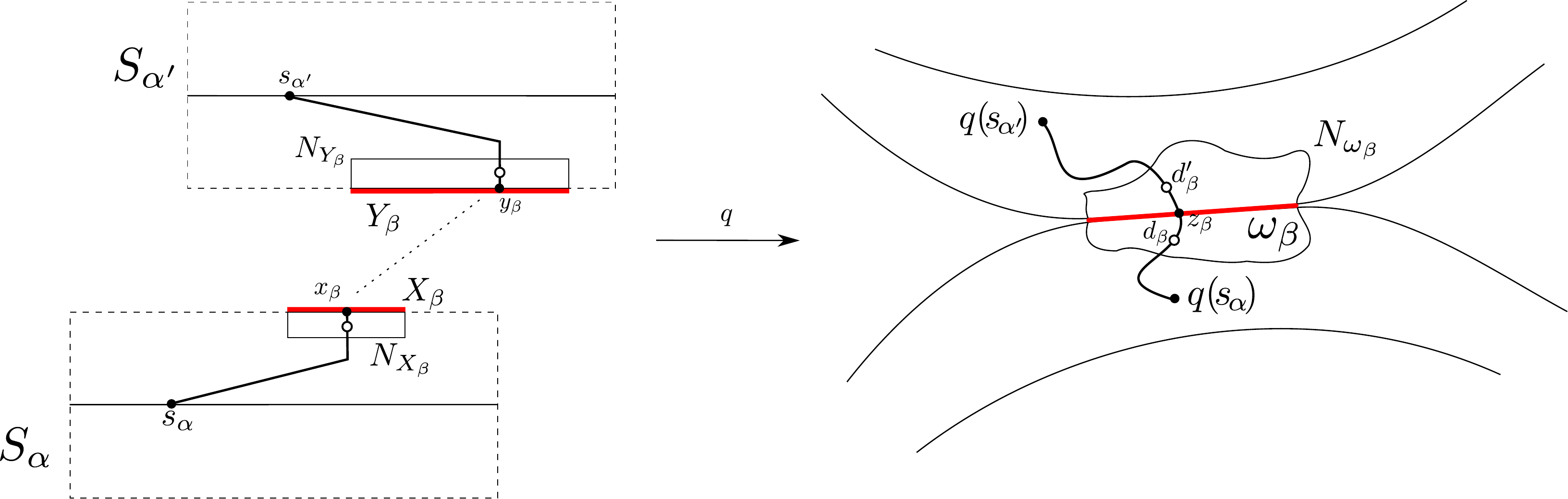}
\caption{The map $\varphi_{\beta}$}\label{fig:path_h}
\end{figure}

\begin{figure}[htbp!]
\includegraphics[width=0.95\textwidth]{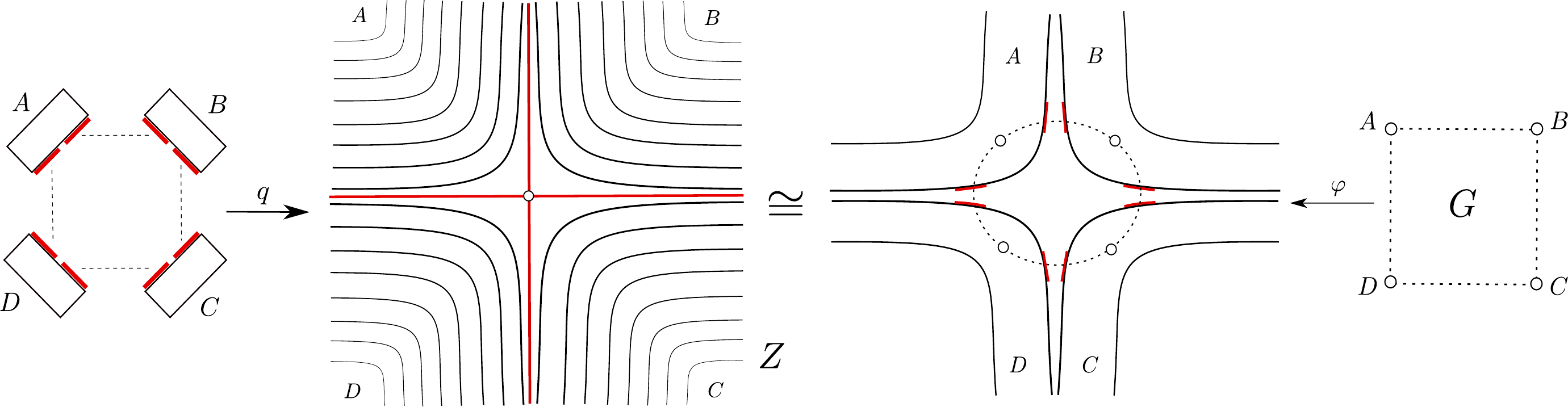}
\caption{Striped structure and the map $\varphi$ on $\RR^2\setminus(0,0)$ for canonical foliation by level sets of the function $f(x,y)=xy$ }\label{fig:xy}
\end{figure}

\begin{figure}[htbp!]
\includegraphics[width=0.8\textwidth]{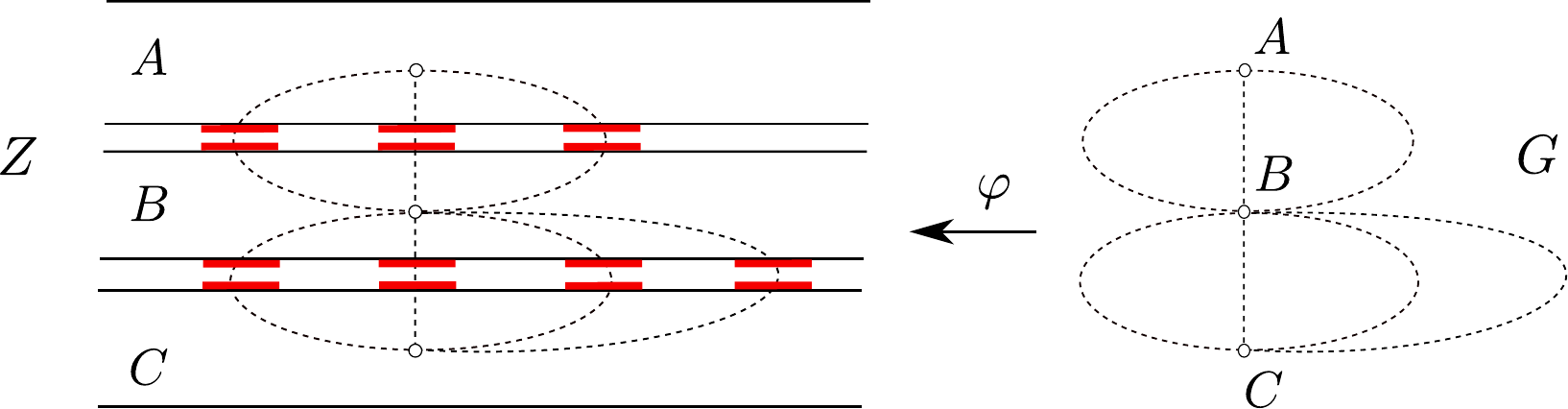}
\caption{}\label{fig:three_strips}
\end{figure}

The aim of the present paper is to prove the following
\begin{theorem}\label{th:phi_Pi1_iso}
Let $(\str,q)$ be a striped surface and $G$ be its graph.
Then for each $x\in G$ the above map $\varphi: G\hookrightarrow \str$ induces an isomorphism of the fundamental groups
\[ \varphi_{*}:\pi_1(G,x)\to \pi_{1}\bigl(\str, \varphi(x)\bigr). \]
In particular, $\varphi$ is a homotopy equivalence.
\end{theorem}

In fact we will establish a more precise statement
\begin{theorem}\label{th:phi_Pi1_iso_groupoids}
Let $(\str,q)$ be a striped surface and $G$ be its graph.
Then there exists a subset $\cutset \subset \str$ such that the map $\varphi: G\hookrightarrow \str$ induces an isomorphism of the corresponding fundamental groupoids
\[ \varphi^{*}:\Pi_1(G,\varphi^{-1}(\cutset))\to \Pi_{1}(\str, \cutset). \]
\end{theorem}
The proof of these theorems will be given in \S\ref{sect:proof:th:phi_Pi1_iso}.
They are based on application of van Kampen theorem for groupoids, see Lemma~\ref{lem-3}.

Notice that when $\str_0$ has only finitely many seams, the result is rather simple and in this case $\varphi(G)$ is a strong deformation retract of $\str$.

On the other hand, if the number of seams is infinite, the graph $G$ might be not a locally finite CW-complex and $\str$ can not be deformed onto the image $\varphi(G)$.

\begin{example}\rm
Let $\str = \RR^2 \setminus (\ZZ\times 0)$.
Then $\str$ has an atlas consisting of two model strips:
\begin{align*}
S_{0} &= \RR \times (-1,1) \ \bigcup \ \mathop{\cup}\limits_{n\in\ZZ} (n,n+1) \times \{1\},  \\
S_{1} &= \RR \times (-1,1) \ \bigcup \ \mathop{\cup}\limits_{n\in\ZZ} (n,n+1) \times \{-1\}.
\end{align*}
and its graph has two vertices connected with countably many edges, see Figure~\ref{fig:inf_loops}.
In this case $G$ is not locally finite at its vertices and therefore it has no countable local base, whence $\varphi: G \to \str$ is not an embedding.
\begin{figure}[htbp!]
\includegraphics[height=3cm]{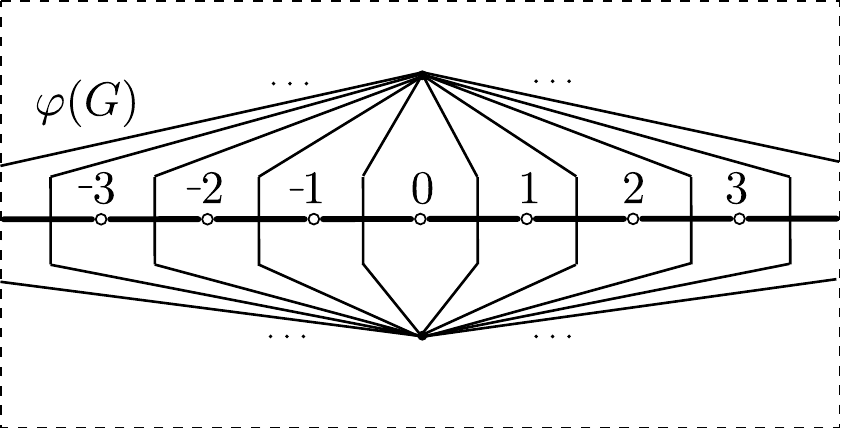}
\caption{}\label{fig:inf_loops}
\end{figure}
\end{example}

Let us also mention the following statement describing certain topological properties of striped surfaces.
\begin{lemma}\label{lm:surf_prop}
Let $\str$ be a connected non-compact surface.
Then $\str$ has the homotopy type of an aspherical CW-complex, and $\pi_1\str$ is free.
\end{lemma}
\begin{proof}
For the fact that every separable manifold has homotopy type of a CW-comples see~\cite[Corollary~5.7.]{LundellWeingram:CW:1969}.

Further, let $p:\tilde{\str}\to\str$ be the universal covering of $\str$.
Then by~\cite[Corollary~1.8]{Epstein:AM:1966}, the interior $\Int{\tilde{\str}}$ is homeomorphic to $\RR^2$.
Moreover, the inclusion $\Int{\tilde{\str}} \subset \tilde{\str}$ is a homotopy equivalence due to existence of collars of the boundary of metrizable manifolds, \cite[Theorem~2]{Brown:AM:1962}.
Hence $p$ induces isomorphisms $0=\pi_k\Int{\tilde{\str}}\cong \pi_k \tilde{\str} \cong \pi_k \str$, $k\geq2$.
This, by definition, means aspherity of $\str$.

Finally, the statement that $\pi_1\str$ is free, is proved in distinct sources, e.g.~\cite{Johansson:1931} or~\cite[Theorem~44A]{AhlforsSario:RS:1960}. 
Another proof is that since $\str$ is non-compact, it's cohomological dimension $cd\str\leq 1$, whence by~\cite[Corollary to Theorem~A]{Swan:JA:1969}, $\pi_1\str$ is free.
See also~\cite{Putman:MO:18454} for discussions and~\cite{Putman:2}.
\end{proof}

\section{Fundamental groupoids}\label{sect:fund_groupoids}
In this section we will briefly recall the notion of fundamental groupoid, list some of its properties and formulate van Kampen theorem for groupoids.

\subsection{Small categories}
A \myemph{small category $\catC$} is given by the following data:
\begin{enumerate}[label=$\bullet$]
\item a set $\catCObj$, called \myemph{set of objects};
\item a set $\catCHom{\oX}{\oY}$ for each pair of objects $\oX,\oY\in\catCObj$, called \myemph{set of morphisms between $\oX$ and $\oY$};
\item for each pair of morphisms $f \in \catCHom{\oX}{\oY}$, $g \in \catCHom{\oY}{\oZ}$ there defined a unique morphism $g \circ f \in \catCHom{\oX}{\oZ}$, called \myemph{composition of $f$ and $g$}
\end{enumerate}
such that the following axioms are satisfied:
\begin{enumerate}[leftmargin=*, label=-]
\item
{\bf{associativity:}} for any three morphisms $f \in \catCHom{\oX}{\oY}$, $g \in \catCHom{\oY}{\oZ}$ and $h\in\catCHom{\oZ}{\oW}$ we have
\[
    h \circ (g \circ f) = (h \circ g) \circ f
\]
\item
{\bf{identity:}}
for each object $\oX\in\catCObj$ there exists a morphism $\id_A \in \catCHom{\oX}{\oX}$, called the \myemph{identity}, such that for each $f \in \catCHom{\oX}{\oY}$ we have
\[
    f \circ \id_{\oX} = \id_{\oY} \circ f = f.
\]
\end{enumerate}

\subsection{Functors}
A \myemph{functor} $F\colon \catC \to \catD$ from category $\catC$ to category $\catD$ is the following collection of maps (usually denoted by the same letter $F$):
\begin{enumerate}[label={\rm\arabic*)}]
\item
a map $F:\catCObj \to \catDObj$, associating to each object $\oX\in\catC$ an object $F(\oX) \in \catD$,

\item
a map $F:\catCHom{\oX}{\oY} \to \catDHom{F(\oX)}{F(\oY)}$ for every pair of objects $\oX,\oY\in\catC$,
\end{enumerate}
such that the following conditions are satisfied:
\begin{enumerate}[label={\rm\alph*)}]
\item
$F(\id_{\oX}) = \id_{F(\oX)}$ for each $\oX \in \catCObj$;

\item
$F(g\circ f)=F(g)\circ F(f)$ for any morphisms $f,g$ in category $\catC$.
\end{enumerate}

\subsection{Coequalizers}
Let $f,g:\oX \to \oY$ be two morphisms.
Then a \myemph{coequalizer} of $f,g$ is a morphism $h:\oY\to\oZ$ such that
\begin{enumerate}[label={\rm\arabic*)}]
\item
$h\circ f = h\circ g$;

\item
for any other morphism $h':\oY\to\oZ'$ with $h'\circ f = h'\circ g$ there exists a unique morphism $q:\oZ \to \oZ'$ such that $h' = q \circ h$.
\end{enumerate}
In other words, we have the following diagram:
\[
    \xymatrix{\oX \ar@/^/[r]^-{f} \ar@/_/[r]_-{g}& \oY \ar[r]^-{h} \ar[rd]_-{h'} & \oZ \ar[d]^-{q}\\
                                        &			    & \oZ'}
\]
which is commutative except for paths $f,g$ connecting $X$ and $Y$.

\subsection{Groupoids}
A morphism $f \in \catCHom{\oX}{\oY}$ is called an \myemph{isomorphism} whenever there exists $g \in \catCHom{\oY}{\oX}$ such that $g \circ f = \id_{\oX}$ and $f \circ g = \id_{\oY}$.

A \myemph{groupoid} is a category in which all morphisms are isomorphisms.
Moreover, a morphism of groupoids is just a functor between the corresponding categories.

\subsection{Example}
For any set $\oX$ the cartesian product $\oX\times\oX$ has a natural groupoid structure.

Evidently, if $f:\oX\to\oY$ is a bijection of sets, then the induced mapping
\[
    f\times f: \oX\times\oX\to\oY\times\oY,
    \qquad
    (f\times f)(a,b) = \bigl(f(a),f(b)\bigr)
\]
is an isomorphism of groupoids.

\subsection{Fundamental groupoid}
Let $\oX$ be a topological space and $\qX\subset \oX$ be a subset.
Let also $I = [0,1]$, and $\CIX$ be the set of continuous paths with ends at $\qX$.
Say that two paths $\alpha, \beta \in \CIX$ are equivalent ($\alpha\sim\beta$) if
\begin{enumerate}
\item $\alpha(0)=\beta(0)$, $\alpha(1)=\beta(1)$;
\item $\alpha$ is homotopic to $\beta$ relatively $\{0,1\}$.
\end{enumerate}
Then the set $\grpXXp = \CIX/\sim$ of the corresponding equivalence classes is called the \myemph{fundamental groupoid of pair $(\oX,\qX)$} and the set $\qX$ is \myemph{the set of its base points}.

The equivalence class of a path $\alpha\in\CIX$ will be denoted by $\ecl{\alpha}$.
For a class $q \in \grpXXp$ we will denote by $q(0)$ and $q(1)$ the common start and end-points of all representatives of $q$.

Then $\grpXXp$ admits a natural ``partial'' operation of composition of paths turning it into groupoid.
Evidently, if $\qX$ consists of a unique point, then the multiplication is defined for all elements of $\grpXXp$, and in this case $\grpXXp$ is the same as the fundamental group of $\oX$ at the point $\qX$.

Notice also that we have a natural map
\begin{equation}\label{equ:proj_to_ends}
    \eprjX: \grpXXp \to \qX \times \qX,
    \qquad
    \eprjX(q) = \bigl(q(0), q(1)\bigr),
\end{equation}
associating to each homotopy class of paths its start and end-points.

Moreover, every continuous map of pairs $f\colon (\oX,\qX)\to (\oY,\qY)$ induces the following commutative diagram:
\begin{equation}\label{equ:simply_conn_spaces}
\aligned
\xymatrix{
    \grpXXp \ar[r]^-{f^*} \ar[d]_-{\eprjX}& \grpYYp \ar[d]^-{\eprjY} \\
     \qX \times \qX \ar[r]^-{f \times f}  &  \qY \times \qY.}
\endaligned
\end{equation}

The following simple lemma is left for the reader:
\begin{lemma}\label{lm:Pi1_prj_XX}
\begin{enumerate}[label={\rm(\alph*)}, leftmargin=*]
\item\label{enum:Pi1_prj_XX:1} The map~\eqref{equ:proj_to_ends} is a morphism of groupoids.
\item\label{enum:Pi1_prj_XX:2} If $\oX$ is path connected and simply connected, then $\eprjX$ is an isomorphism.
\item\label{enum:Pi1_prj_XX:3} Suppose $f:(\oX,\qX)\to(\oY,\qY)$ is a continuous map, $\oX$ and $\oY$ are path connected and simply connected, and the restriction map $f|_{\qX}\colon \qX \to \qY$ is a bijection.
Then all morphisms in~\eqref{equ:simply_conn_spaces} are isomorphisms.
\end{enumerate}
\end{lemma}

\subsection{Coproducts}
Let $\{\oX_{\lambda}\}_{\lambda\in\Lambda}$ be a collection of objects of a category $\catC$.
Then their \myemph{coproduct} is an object $\oY = \UnionLambda \oX_{\lambda}$ together with a collection of morphisms $i_{\lambda}: \oX_{\lambda} \to \oY$ such that for any other collection of morphisms $\{f_{\lambda}: \oX_{\lambda} \to \oZ\}_{\lambda\in\Lambda}$, to the same object $\oZ$ there exists a unique morphism $f\colon \oY \to \oZ$ such that $f_{\lambda} = f \circ i_{\lambda}$.
In other words, for each $\lambda\in\Lambda$ the following diagram is commutative:
\[
\xymatrix{
    \oX_{\lambda}  \ar[rr]^-{i_{\lambda}} \ar[rrd]_-{f_{\lambda}} && \UnionLambda \oX_{\lambda} \ar[d]^-{f=\UnionLambda f_{\lambda}} \\
                                    && \oZ
}
\]

Consider examples.
\begin{enumerate}[wide, label={\rm\alph*)}]
\item
In the category of sets a coproduct is just a disjoint union of sets.

\item
Let $\{ G_{\lambda}\}_{\lambda\in\Lambda}$ be a collection of groupoids with multiplications
\[ \mu_{\lambda}: G_{\lambda} \times G_{\lambda} \supset Q_{\lambda} \to G_{\lambda},\]
and sets of objects $E_{\lambda}$.
Then their disjoint union
\[ G = \UnionLambda  G_{\lambda}\]
has a natural groupoid structure in which the partial multiplication
\[
    \mu: G \times G \supset  \UnionLambda  Q_{\lambda} \to G,
\]
is defined by $\mu(a,b) = \mu_{\lambda}(a,b)$, when $(a,b)\in Q_{\lambda}$.
This groupoid is the coproduct in the category of groupoids.

Notice also that if $\{f_{\lambda}: G_{\lambda} \to H\}_{\lambda\in\Lambda}$ is a collection of morphisms of groupoids, then the induced morphism $f=\UnionLambda f_{\lambda}:G \to H$ is given by $f(a) = f_{\lambda}(a)$ if $a\in G_{\lambda}$.

\item\label{groupoid_iso}
Let $\{ (\oX_{\lambda}, \qX_{\lambda})\}_{\lambda\in\Lambda}$ be a collection of pairs of topological spaces, $\oX =  \UnionLambda  \oX_{\lambda}$ and $\qX = \UnionLambda  \qX_{\lambda}$ be the corresponding coproducts of sets.
For each $\lambda\in\Lambda$ the natural inclusion of pairs $\incl{\lambda}: (\oX_{\lambda}, \qX_{\lambda}) \subset (\oX,\qX)$ induces the corresponding morphism of groupoids
\[
    \inclGrp{\lambda}: \grpXlXpl \to \grpXXp.
\]
Hence by coproduct property, we have a unique morphism $\zeta$ making the following diagram commutative:
\begin{equation}\label{equ:coproduct_iso}
\aligned
\xymatrix{
    \grpXlXpl \ar@{^(->}[rr]^-{i_{\lambda}}  \ar[rrd]_-{\inclGrp{\lambda}} && \UnionLambda \grpXlXpl \ar[d]^{\zeta = \UnionLambda \inclGrp{\lambda}}  \\
    && \grpXXp
}
\endaligned
\end{equation}
where $i_{\lambda}$ is a natural morphism into coproduct.
We claim that $\zeta$ is an isomorphism, i.e. it is bijective.

Let us show that $\zeta$ is \myemph{surjective}.
Indeed, let $\alpha:(I,\partial I) \to (\oX,\qX)$ be a continuous map.
We should show that $\ecl{\alpha} = \zeta(q)$ for some $q \in \UnionLambda \grpXlXpl$.
Since $I$ is path-connected and $\oX_{\lambda} \cap \oX_{\nu} = \varnothing$ for $\lambda\not=\nu$ it follows that $\alpha(I)$ is contained in some $\oX_{\lambda}$.
Hence $\alpha = \incl{\lambda}\circ \beta$ for some path $\beta:(I,\partial I) \to (\oX_{\lambda},\qX_{\lambda})$.
Therefore, $\ecl{\alpha} = \inclGrp{\lambda}\bigl(\ecl{\beta}\bigr) = \zeta\circ i_{\lambda}\bigl( \ecl{\beta}\bigr) = \zeta\bigl( \ecl{ i_{\lambda}\circ\beta } \bigr)$.

Let us prove that $\zeta$ is \myemph{injective}.
Let $p,q \in \UnionLambda \grpXlXpl$ be two classes such that $\zeta(p)=\zeta(q)$.
By definition, $p = i_{\lambda}([\alpha])$ and $q = i_{\nu}([\beta])$ for some paths
\[
    \alpha:(I,\partial I) \to (\oX_{\lambda},\qX_{\lambda}), \qquad
    \beta:(I,\partial I) \to (\oX_{\nu},\qX_{\nu}).
\]
Then
\[
    \inclGrp{\lambda}\bigl(\ecl{\alpha}\bigr) =   \zeta\bigl( i_{\lambda}([\alpha]) \bigr) = \zeta(p) = \zeta(q) = \zeta\bigl( i_{\nu}([\beta]) \bigr)
    = \inclGrp{\nu}\bigl(\ecl{\beta}\bigr)
\]
which means that the paths $\incl{\lambda}\circ\alpha, \incl{\lambda}\circ\beta: (I,\partial I)\to (\oX, \qX)$ are homotopic relatively their ends.
Hence $\lambda = \nu$.
Moreover, let $H:[0,1]\times[0,1]\to\oX$ be the corresponding homotopy between $H_0 = \alpha$ and $H_1=\beta$.
Since $[0,1]\times[0,1]$ is path-connected, it follows that $H([0,1]\times[0,1]) \subset \oX_{\lambda}$, and thus $H$ is a homotopy between $\alpha$ and $\beta$ in $\oX_{\lambda}$ relatively their ends.
This means that $\lambda=\nu$, $\ecl{\alpha}=\ecl{\beta} \in \grpXlXpl$, whence $p = i_{\lambda}(\ecl{\alpha}) = i_{\nu}(\ecl{\beta}) = q$.
Thus $\zeta$ is injective.
\end{enumerate}

\subsection{van Kampen theorem for groupoids}
Let $\Lambda$ be a set.
It will be convenient to consider a category $\catC$ whose objects are triples $\bigl( \oX, \qX, \coverU \bigr)$ where
\begin{enumerate}
    \item $\oX$ is a topological space;
    \item $\qX \subset\oX$ a subset;
    \item $\coverU =\{\Ucov{\lambda} \}_{\lambda \in \Lambda}$ is an open cover of $\oX$ enumerated by the same set of indices $\Lambda$;
\end{enumerate}
and morphisms between triples $\bigl( \oX, \qX, \coverU =\{\Ucov{\lambda} \}_{\lambda \in \Lambda}\bigr)$ and $\bigl(\oY,\qY,\coverV=\{\Vcov{\lambda}\}_{\lambda \in \Lambda}\bigr) \in \catC$ are continuous maps of pairs $f\colon (\oX,\qX) \to (\oY,\qY)$ such that $f\bigl( \Vcov{\lambda} \bigr) \subset \Ucov{\lambda}$ for all $\lambda\in\Lambda$.

Given a triple $\bigl( \oX, \qX, \coverU \bigr) \in \catC$,
for each $n$-tuple $\nu = (\nu_1,\ldots,\nu_n) \in \Lambda^n$ put
\begin{align*}
    \Ucov{\nu}   &:= \Ucov{\nu_1}\cap\cdots\cap U^{\nu_n}, &
    \zUcov{\nu} &:= \Ucov{\nu}\cap \qX.
\end{align*}
In particular, for $n=2$ and $\nu = (j,k) \in A^2$ we have the following two morphisms of fundamental groupoids:
\begin{align*}
    &a_{jk}:\grp{\Ucov{j}\cap \Ucov{k}}{\zUcov{(j,k)}} \to \grpUcov{j}, &
    &b_{jk}:\grp{\Ucov{j}\cap \Ucov{k}}{\zUcov{(j,k)}} \to \grpUcov{k}
\end{align*}
induced by natural inclusions of $\Ucov{j}\cap U^k$ into $U^j$ and $U^k$ respectively.

These morphisms yield morphisms of the corresponding coproducts:
\[
    a,b: \UnionJK  \grpUcov{(j,k)} \to \bigsqcup\limits_{\lambda \in \Lambda} \grpUcov{\lambda}
\]

Similarly, the inclusion $(\Ucov{j},\zUcov{j}) \subset (\oX, \qX)$ induces a morphism
\[ c_j: \grpUcov{j} \to \grpXXp. \]

Then the \myemph{$\varPi_1$-diagram of the cover $\coverU $} is the following diagram:
\begin{equation}\label{equ:diag21}
\xymatrix{
    \UnionJK  \grpUcov{(j,k)}
    \ar@/^/[r]^-a \ar@/_/[r]_-b&
    \bigsqcup\limits_{\lambda \in \Lambda}
    \grpUcov{\lambda} \ar[r]^-c & \grpXXp.
}
\end{equation}

\begin{theorem}[van Kampen theorem for fundamental groupoids, \cite{BrownSalleh:AM:1984}]\label{th:vanKampenTh}
Suppose that an open cover $\coverU =\{\Ucov{\lambda}\}_{\lambda \in \Lambda}$ of $\oX$ has the following property:
\begin{enumerate}[label={$(*)$}]
\item\label{enum:vanKampenTh}
 a subset $\qX \subset \oX$ meets each path-component of each non-empty two-fold and three-fold intersection of distinct sets of $\coverU$.
\end{enumerate}
Then in the $\varPi_1$-diagram~\eqref{equ:diag21} of the cover $\coverU $ the morphism $c$ is a coequalizer for morphisms $a,b$ in the  category of fundamental groupoids.
\end{theorem}

\subsection{$\varPi_1$-diagram for covers by simply connected sets}
Let
\[
    f: \bigl( \oX, \qX, \{\Ucov{\lambda} \}_{\lambda \in \Lambda}\bigr) \to  \bigl(\oY,\qY,\{\Vcov{\lambda}\}_{\lambda \in \Lambda}\bigr)
\]
be a morphism in the category $\catC$.
Then it is evident that $f$ induces the following diagram:
\begin{equation}\label{equ:diag22}
\aligned
\xymatrix{
    \UnionJK  (\Ucov{(j,k)},\zUcov{(j,k)})
        \ar@/^/[r]^-{a} \ar@/_/[r]_-{b} \ar[d]^-{\indxx{f}} &
    \UnionLambda (\Ucov{\lambda},\zUcov{\lambda})
    \ar[r]^-{c}  \ar@<1ex>[d]^-{\indx{f}}
    &
    (\oX,\qX) \ar@<1ex>[d]^-f \\
    \UnionJK  (\Vcov{(j,k)},\zVcov{(j,k)})
        \ar@/^/[r]^-{a'} \ar@/_/[r]_-{b'}  &
    \UnionLambda(\Vcov{\lambda},\zVcov{\lambda}) \ar[r]^-{c'}
    &
    (\oY,\qY)
}
\endaligned
\end{equation}
where $\indx{f} = \UnionLambda f|_{\Vcov{\lambda}}$ and $\indxx{f} = \UnionJK f|_{\Vcov{(j,k)}}$.
This diagram is commutative except for the left square in which we have only the following identities:
\begin{align}\label{equ:comm_left_square}
    f_1\circ a  &= a' \circ f_2, &
    f_1 \circ b &= b'\circ f_2.
\end{align}

Since $\varPi_1$ is a functor from the category of sets to the category of groupoids, \eqref{equ:diag22} implies the following diagram combined from two $\varPi_1$-diagrams:
\begin{equation}\label{equ:diag23}
\aligned
\xymatrix{
    \UnionJK  \grpUcov{(j,k)}
        \ar@/^/[r]^-a \ar@/_/[r]_-b  \ar@<1ex>[d]^-{\indxx{f}}&
    \UnionLambda \grpUcov{\lambda}
        \ar[r]^-c  \ar@<1ex>[d]^-{\indx{f}} &
    \grpXXp  \ar@<1ex>[d]^-{f} \\
    \UnionJK  \grpVcov{(j,k)}
        \ar@/^/[r]^-{a'} \ar@/_/[r]_-{b'} &
    \UnionLambda \grpVcov{\lambda}
        \ar[r]^-{c'}
    &  \grpYYp
}
\endaligned
\end{equation}
Here we used the isomorphism of groupoids~\eqref{equ:coproduct_iso}.
Notice that diagram~\eqref{equ:diag23} is also commutative except for left square in which only the identities~\eqref{equ:comm_left_square} hold.

The following lemma plays a key role in proving the basic theorem.
\begin{lemma}\label{lem-3}
Let $f: \bigl( \oX, \qX, \{\Ucov{\lambda} \}_{\lambda \in \Lambda}\bigr) \to  \bigl(\oY,\qY,\{\Vcov{\lambda}\}_{\lambda \in \Lambda}\bigr)$ be a morphism in the category $\catC$.
Suppose that
\begin{enumerate}[label={\rm(\arabic*)}, leftmargin=*, itemsep=1ex]
\item\label{enum:cov_lemma:1}
for each $\lambda \in \Lambda$ spaces $\Ucov{\lambda}$ and $\Vcov{\lambda}$ are simply connected;

\item\label{enum:cov_lemma:2}
for each $\lambda \in \Lambda$ the restriction
$f|_{\zUcov{\lambda}}\colon \zUcov{\lambda} \equiv \Ucov{\lambda}\cap\qX \, \to \, \zVcov{\lambda} \equiv \Vcov{\lambda}\cap\qY$ is a bijection;

\item\label{enum:cov_lemma:3}
the open covers $\{\Ucov{\lambda}\}_{\lambda \in \Lambda}$ and $\{\Vcov{\lambda}\}_{\lambda \in \Lambda}$ satisfy condition~\ref{enum:vanKampenTh} of van Kampen Theorem~\ref{th:vanKampenTh}, that is $\qX$ (resp.~$\qY$) meets each path component of each non-empty two-fold and three-fold intersections of elements of $\coverU$ (resp.~$\coverV$).
\end{enumerate}
Then the induced morphism $f^{*}: \grpXXp \to  \grpYYp$ of fundamental groupoids from diagram~\eqref{equ:diag23} is an isomorphism.
\end{lemma}
\begin{proof}
Conditions~\ref{enum:cov_lemma:1} and~\ref{enum:cov_lemma:2} mean that for each $\lambda\in\Lambda$ the restriction map
\[ f|_{\Ucov{\lambda}}:(\Ucov{\lambda},\zUcov{\lambda}) \to (\Vcov{\lambda},\zVcov{\lambda}) \]
satisfies condition~\ref{enum:Pi1_prj_XX:3} of Lemma~\ref{lm:Pi1_prj_XX}.
Therefore the induced morphism of groupoids $f_{\lambda}:\grpUcov{\lambda} \to \grpVcov{\lambda}$ is an isomorphism.
Hence so is the middle vertical morphism $\indx{f}$ of diagram~\eqref{equ:diag23}.

Let $\indx{g}:\UnionLambda \grpVcov{\lambda} \to \UnionLambda \grpUcov{\lambda}$ be the isomorphism inverse to $\indx{f}$.
Denote $k = c \circ \indx{g}$, see~\eqref{equ:diag23_inv}:
\begin{equation}\label{equ:diag23_inv}
\aligned
    \xymatrix{
        \UnionJK  \grpUcov{(j,k)}
            \ar@/^/[r]^-a \ar@/_/[r]_-b  \ar@<1ex>[d]^-{\indxx{f}}&
        \UnionLambda \grpUcov{\lambda}
            \ar[r]^-c  \ar@<1ex>[d]^-{\indx{f}} &
        \grpXXp  \ar@<1ex>[d]^-{f} \\
        \UnionJK  \grpVcov{(j,k)}
            \ar@/^/[r]^-{a'} \ar@/_/[r]_-{b'} &
        \UnionLambda \grpVcov{\lambda}
            \ar[r]^-{c'} \ar@<1ex>[u]^-{\indx{g}}   \ar[ur]^-{k}
        &  \grpYYp  \ar@<1ex>[u]^-{g}
    }
\endaligned
\end{equation}
Then $k\circ a' = k \circ b'$, whence by van Kampen Theorem~\ref{th:vanKampenTh} there exists a \myemph{unique} morphism $g:\grpYYp \to \grpXXp$ such that $k = g \circ c'$.
Since $\indx{f}$ and $\indx{g}$ are inverse each to other, it follows that $f$ and $g$ must also be inverse each to other and therefore $f$ is an isomorphism.
\end{proof}

\section{Proof of Theorem~\ref{th:phi_Pi1_iso_groupoids}}\label{sect:proof:th:phi_Pi1_iso}
Let $(\str, q)$ be a striped surface, $G$ its graph, and $\varphi:G \to \str$ the continuous injective map defined in \S\ref{sect:inject_G_Z}.
We should construct a subset $\cutset \subset \str$ such that the map $\varphi: G\hookrightarrow \str$ induces an isomorphism of the corresponding fundamental groupoids $\varphi_{*}:\Pi_1(G,\varphi^{-1}(\cutset))\to \Pi_{1}(\str, \cutset)$.

In fact we will also define a special open cover $\coverV$ of $\str$, and then consider a cover $\coverU = \varphi^{-1}(\coverV)$ of $G$ consisting of inverse of elements of $\coverV$.
Then $\varphi$ will evidently induce a morphism of triples $\varphi:(G,\cutGset, \coverU) \to (\str,\cutset, \coverV)$ in the category $\catC$, and we will show that conditions~\ref{enum:cov_lemma:1}-\ref{enum:cov_lemma:3} of Lemma~\ref{lem-3} are satisfied.
This will imply that $\varphi_{*}$ is an isomorphism.

\subsection*{An open cover $\coverV$ of $\str$}
Let $S$ be a model strip and $X = (a,b) \times \eps  \subset \partial S$, $\eps\in\{\pm1\}$, be a boundary interval.
Then the following \myemph{open} neighborhood of $X$ in $S$:
\[
\nb{X} =
\begin{cases}
    (a,b)\times(0.8, 1], & \eps = 1, \\
    (a,b)\times[-1, 0.8), & \eps = -1.
\end{cases}
\]
will be called the \myemph{standard} neighborhood of $X$.

It is evident, that standard neighborhoods of boundary intervals are mutually disjoint.
Hence the family of boundary intervals of $S$ is \myemph{discrete}
\footnote{Recall that a collections of subsets $\{Q_i\}_{i\in\Lambda}$ of a topological space $X$ is called \myemph{discrete}, if for each $i\in\Lambda$ there exists an open neighborhood $U_i$ of $Q_i$ such that $U_i\cap U_j=\varnothing$ for $i\not=j\in\Lambda$.}
and in particular locally finite.
Therefore, the union of any number of boundary intervals is a closed set.

By a \myemph{standard} neighborhood $\nbs{\beta}$ of a seam $\seam{\beta}$ we will mean the union of images of standard neighborhoods of $X_{\beta}$ and $Y_{\beta}$ (see Figure~\ref{fig:path_h}):
\[
\nbs{\beta} = q(\nb{X_{\beta}}) \cup q(\nb{Y_{\beta}}).
\]
Then $\nbs{\beta}$ is open in $\str$.
Moreover, $\nbs{\beta} \cap \nbs{\beta'} = \varnothing$ for $\beta\not=\beta'$, whence the family of seams is discrete and therefore locally finite.
In particular, the union of any collections of seams is also closed.

Furthermore, for $\alpha\in A$ put
\[
    \nbStr{\alpha} := q\bigl(S_{\alpha} \setminus (X\cup Y)\bigr).
\]
Thus $\nbStr{\alpha}$ is image of a model strip $S_{\alpha}$ without boundary intervals corresponding to seams.
It follows that $\nbStr{\alpha}$ is open in $\str$.
Hence we get the following open cover of $\str$.
\[
    \coverV := \{ \nbStr{\alpha} \}_{\alpha\in A} \ \cup \ \{ \nbs{\beta}\}_{\beta\in B}.
\]

\subsection*{The set $\cutset$}
Notice that $\nbs{\beta} \setminus \seam{\beta}$ consists of two connected components each homeomorphic to an open rectangle and the image $\varphi(G)$ intersects each of those components.
Choose any two points
\begin{align*}
  &d_{\beta} \in \varphi(G) \cap \bigl( q(\nb{X_{\beta}}) \setminus \seam{\beta} \bigr), &
  &d'_{\beta} \in \varphi(G) \cap \bigl( q(\nb{Y_{\beta}}) \setminus \seam{\beta} \bigr).
\end{align*}
belonging to distinct components of $\nbs{\beta} \setminus \seam{\beta}$.
In Figure~\ref{fig:path_h} such points are denoted by circles.

Let $J'$ be the set of isolated vertices of $G$.
Then every strip $S_{\alpha}$ with $\alpha\in J$ is not glued to any other strips, and no boundary intervals of $S$ are glued together.
Let also $K' = \{ s_{\alpha} = (0,0) = \varphi(\alpha) \in S_{\alpha} \mid \alpha \in J' \} \subset \str$ be the set of origins of such strips in $\str$.
Put
\[
    \cutset = K' \ \cup \ \{d_{\beta}, d'_{\beta} \mid \beta \in B\},
\]
and
\begin{align*}
    \cutGset &= \varphi^{-1}(\cutset), &
    \lbStr{\alpha} &:= \varphi^{-1}(\nbStr{\alpha}), &
    \lbs{\beta}    &:= \varphi^{-1}(\nbs{\beta}).
\end{align*}
Thus $\lbs{\beta}$ is an open arc in some $1$-cell of $G$ containing two points $\varphi^{-1}(d_{\beta})$ and $\varphi^{-1}(d'_{\beta})$, while $\lbStr{\alpha}$ is a ``star''-neighborhood of the vertex $\alpha\in G^{0} = A$ such that each edge of $\lbStr{\alpha}$ contains a unique point $\varphi^{-1}(d_{\beta})$ or $\varphi^{-1}(d'_{\beta})$ for some $\beta\in B$.

It follows that
\[ \coverU := \{ \lbStr{\alpha} \}_{\alpha\in A} \ \cup \ \{ \lbs{\beta}\}_{\beta\in B}\]
is an open cover of $G$, and $\varphi:(G,\cutGset, \coverU) \to (\str,\cutset, \coverV)$ induces a morphism in the category $\catC$.

\subsection*{Verification of conditions of Lemma~\ref{lem-3}}

\ref{enum:cov_lemma:1}
Evidently, the elements of $\coverU$ and $\coverV$ are even contractible and therefore they are simply connected.

\ref{enum:cov_lemma:2}
Since $\varphi$ is injective, it follows that for any subset $Q\subset \varphi(G) \subset \str$, we have that $\varphi|_{\varphi^{-1}(Q)}: \varphi^{-1}(Q) \to Q$ is a bijection.
In particular, so are the restrictions
\begin{align*}
&\varphi: \cutGset \cap \lbStr{\alpha}  \to \cutset \cap \nbStr{\alpha}, &
&\varphi: \cutGset \cap \lbs{\beta}  \to \cutset \cap \nbs{\beta}
\end{align*}
for each $\alpha\in A$ and $\beta\in B$.

\ref{enum:cov_lemma:3}
First notice that $\cutset$ (resp.~$\cutGset$) meets every path component of $\str$ (resp.~$G$).
Furthermore,
$\nbStr{\alpha} \cap \nbStr{\alpha'}  =  \nbs{\beta}  \cap  \nbs{\beta'} =
  \lbStr{\alpha} \cap \lbStr{\alpha'}  =  \lbs{\beta}  \cap  \lbs{\beta'}  =
 \varnothing$ for $\alpha \not=\alpha'\in A$ and $\beta \not=\beta'\in B$, which implies that all three-fold intersections of elements of $\coverV$ and $\coverU$ are empty.

Also, $\nbStr{\alpha} \cap \nbs{\beta}  \not=\varnothing$ iff either $X_{\beta}$ or $Y_{\beta}$ or both of them are contained in $\partial S_{\alpha}$.
In this case each connected component of $\nbStr{\alpha} \cap \nbs{\beta}$ contains either $d_{\beta}$ or $d'_{\beta}$.
It follows that each connected component of $\lbStr{\alpha} \cap \lbs{\beta}$ contains either $\varphi^{-1}(d_{\beta})$ or $\varphi^{-1}(d'_{\beta})$.
Thus $\cutset$ (resp.~$\cutGset$) meets all path components of all two-fold intersections of elements of $\coverV$ (resp.~$\coverU$).
Hence all conditions of Lemma~\ref{lem-3} holds, whence $\varphi_{*}$ is an isomorphism of groupoids.
\qed

\section{Proof of Theorem~\ref{th:phi_Pi1_iso}}\label{sect:proof:th:phi_Pi1_iso_1}
By Theorem~\ref{th:phi_Pi1_iso_groupoids}, $\varphi$ yields an isomorphism of fundamental groupoids.
In particular, if $x\in\cutGset$, then $\varphi$ also induces an isomorphism of the fundamental groups $\pi_1(G,x) \to \pi_1(\str,\varphi(x))$.
Since $\cutset$ (resp.~$\cutGset$) meets every path component of $\str$ (resp.~$G$), it follows that $\varphi$ induces isomorphism of fundamental groups are each point $x\in G$.

Notice that every connected component of $G$ is covered by at most countable tree, and therefore $G$ is aspherical.
Moreover, $\str$ is also aspherical by Lemma~\ref{lm:surf_prop}.
Hence $\varphi$ induces isomorphisms between all the corresponding homotopy groups of $G$ and $\str$.
Moreover, since $\str$ has the homotopy type of an infinite CW-complex, see Lemma~\ref{lm:surf_prop}, we have from the Whitehead theorem that $\varphi$ is a homotopy equivalence between all corresponding connected components of $G$ and $\str$.
\qed

\bibliographystyle{plain}
\bibliography{biblio}

\begin{thebibliography}{10}

\bibitem{AhlforsSario:RS:1960}
Lars~V. Ahlfors and Leo Sario.
\newblock {\em Riemann surfaces}.
\newblock Princeton Mathematical Series, No. 26. Princeton University Press,
  Princeton, N.J., 1960.

\bibitem{Boothby:AJM_1:1951}
William~M. Boothby.
\newblock The topology of regular curve families with multiple saddle points.
\newblock {\em Amer. J. Math.}, 73:405--438, 1951.

\bibitem{Boothby:AJM_2:1951}
William~M. Boothby.
\newblock The topology of the level curves of harmonic functions with critical
  points.
\newblock {\em Amer. J. Math.}, 73:512--538, 1951.

\bibitem{Brown:AM:1962}
Morton Brown.
\newblock Locally flat imbeddings of topological manifolds.
\newblock {\em Ann. of Math. (2)}, 75:331--341, 1962.

\bibitem{BrownSalleh:AM:1984}
Ronald Brown and Abdul~Razak Salleh.
\newblock A van {K}ampen theorem for unions on nonconnected spaces.
\newblock {\em Arch. Math. (Basel)}, 42(1):85--88, 1984.

\bibitem{Epstein:AM:1966}
D.~B.~A. Epstein.
\newblock Curves on {$2$}-manifolds and isotopies.
\newblock {\em Acta Math.}, 115:83--107, 1966.

\bibitem{JenkinsMorse:AJM:1952}
James Jenkins and Marston Morse.
\newblock Contour equivalent pseudoharmonic functions and pseudoconjugates.
\newblock {\em Amer. J. Math.}, 74:23--51, 1952.

\bibitem{Johansson:1931}
Ingebrigt Johansson.
\newblock Topologische {U}ntersuchungen \"uber unverzweigte
  {\"u}berlagerungsfl\"achen, 1931.

\bibitem{Kaplan:DJM:1940}
Wilfred Kaplan.
\newblock Regular curve-families filling the plane, {I}.
\newblock {\em Duke Math. J.}, 7:154--185, 1940.

\bibitem{Kaplan:DJM:1941}
Wilfred Kaplan.
\newblock Regular curve-families filling the plane, {II}.
\newblock {\em Duke Math. J.}, 8:11--46, 1941.

\bibitem{LundellWeingram:CW:1969}
Albert~T. Lundell and Stephen Weingram.
\newblock {\em The topology of {CW} complexes}.
\newblock The University Series in Higher Mathematics. Van Nostrand Reinhold
  Co., New York, 1969.

\bibitem{MaksymenkoPolulyakh:PIGC:2016}
S.~Maksymenko and E.~Polulyakh.
\newblock One-dimensional foliations on topological manifolds.
\newblock {\em Proc. Int. Geom. Cent.}, 9(2):1--23, 2016.

\bibitem{MaksymenkoPolulyakh:PIGC:2015}
Sergiy Maksymenko and Eugene Polulyakh.
\newblock Foliations with non-compact leaves on surfaces.
\newblock {\em Proceedings of Geometric Center}, 8(3--4):17--30, 2015.

\bibitem{MaksymenkoPolulyakh:MFAT:2016}
Sergiy Maksymenko and Eugene Polulyakh.
\newblock Foliations with all non-closed leaves on non-compact surfaces.
\newblock {\em Methods Funct. Anal. Topology}, 22(3):266--282, 2016.

\bibitem{MaksymenkoPolulyakh:PIGC:2017}
Sergiy Maksymenko and Eugene Polulyakh.
\newblock Characterization of striped surfaces.
\newblock {\em Proc. Int. Geom. Cent.}, 10(2):24--38, 2017.

\bibitem{MaksymenkoPolulyakhSoroka:PIGC:2017}
Sergiy Maksymenko, Eugene Polulyakh, and Yuliya Soroka.
\newblock Homeotopy groups of one-dimensional foliations on surfaces.
\newblock {\em Proc. Int. Geom. Cent.}, 10(1):22--46, 2017.

\bibitem{Morse:TopMeth:1947}
Marston Morse.
\newblock {\em Topological methods in the theory of functions of a complex
  variable}.
\newblock Annals of Mathematics Studies, no. 15. Princeton University Press,
  Princeton, N. J., 1947.

\bibitem{Morse:FM:1952}
Marston Morse.
\newblock The existence of pseudoconjugates on {R}iemann surfaces.
\newblock {\em Fund. Math.}, 39:269--287 (1953), 1952.

\bibitem{Putman:2}
Andrew Putman.
\newblock Spines of manifolds and the freeness of fundamental groups of
  noncompact surfaces.
\newblock https://www3.nd.edu/$\sim$andyp/notes/NoncompactSurfaceFree.pdf.

\bibitem{Putman:MO:18454}
Andy Putman.
\newblock Fundamental groups of noncompact surfaces.
\newblock MathOverflow.
\newblock https://mathoverflow.net/q/18454 (version: 2015-10-27).

\bibitem{Soroka:MFAT:2016}
Yuliya Soroka.
\newblock Homeotopy groups of rooted tree like non-singular foliations on the
  plane.
\newblock {\em Methods Funct. Anal. Topology}, 22(3):283--294, 2016.

\bibitem{Soroka:UMJ:2017}
Yuliya Soroka.
\newblock Homeotopy groups of nonsingular foliations of the plane.
\newblock {\em Ukr. Math. Journ.}, 69(7):1000--1008, 2017.

\bibitem{Swan:JA:1969}
Richard~G. Swan.
\newblock Groups of cohomological dimension one.
\newblock {\em J. Algebra}, 12:585--610, 1969.

\end{thebibliography}

\end{document}